\documentclass{amsart}
\usepackage{amssymb,latexsym}
\theoremstyle{plain}
\newtheorem{theorem}{Theorem}
\newtheorem{corollary}{Corollary}
\newtheorem{proposition}{Proposition}
\newtheorem{lemma}{Lemma}
\theoremstyle{definition}
\newtheorem{definition}{Definition}
\newtheorem{example}{Example}

\newtheorem{remark}{Remark}

\newcommand{\enm}[1]{\ensuremath{#1}}          %

\newcommand{\cal}[1]{\mathcal{#1}}

\newcommand{\Nm}{\mathrm{Num}}

\newcommand{\CC}{\enm{\mathbb{C}}}
\newcommand{\II}{\enm{\mathbb{I}}}

\newcommand{\RR}{\enm{\mathbb{R}}}
\newcommand{\QQ}{\enm{\mathbb{Q}}}

\newcommand{\FF}{\enm{\mathbb{F}}}

\newcommand{\Rr}{\enm{\cal{R}}}

\renewcommand{\phi}{\varphi}
\renewcommand{\theta}{\vartheta}
\renewcommand{\epsilon}{\varepsilon}

\renewcommand{\to}[1][]{\xrightarrow{\ #1\ }}

\newcommand{\old}[1]{}

\date{}

\begin{document}

\title[numerical range]
{On the convexity of numerical range over certain fields}
\author{E. Ballico}
\address{Dept. of Mathematics\\
 University of Trento\\
38123 Povo (TN), Italy}
\email{ballico@science.unitn.it}
\thanks{The author was partially supported by MIUR and GNSAGA of INdAM (Italy).}

\begin{abstract}
Let $L$ be a degree $2$ Galois extension of the field $K$ and $M$ an $n\times n$ matrix with coefficients in $L$. Let $\langle \ ,\ \rangle :
L^n\times L^n\to L$ be the sesquilinear form associated to the involution $\sigma: L\to L$ fixing $K$. This sesquilinear form defines
the numerical range $\mathrm{Num}(M)$ of any $n\times n$ matrix over $L$. In this paper we study the convexity of $\mathrm{Num}(M)$
(under certain assumptions on $K$ and/or $M$). Many of the results are for ordered fields.
\end{abstract}

\maketitle

\section{Introduction}

For any field $F$ let $M_{n,n}(F)$ denote the set of all $n\times n$ matrices with coefficients in $F$. Fix fields $K\subset L$ such that $L$ is a degree $2$ Galois extension of $K$ and call $\sigma$
the generator of the Galois group of the extension $K\subset L$. For any $u=(u_1,\dots ,u_n), v=(v_1,\dots ,v_n)\in L^n$ set
$\langle u,v\rangle := \sum _{i=1}^{n} \sigma (u_i)v_i$. The map $\langle \ ,\ \rangle : L^n\times L^n\to L$ is sesquilinear (linear in the second variable and $\sigma$-linear in the first one) and $\langle u,v\rangle = \sigma (\langle v,u\rangle )$
for all $u, v\in L^n$. For any $M\in M_{n,n}(L)$ and $u\in L$ set $\nu _M(u):= \langle u,Mu\rangle$. We obtain a map $\nu _M: L^n\to L$, call the \emph{numerical map} of $M$.
When $K=\RR$ and $L=\CC$ (and so $\sigma$ is the complex conjugation and $\langle \ ,\ \rangle$ is the usual Hermitian product) the image of the restriction of $\nu _M$ to $C_n(1):= \{z\in L^n\mid \langle z,z\rangle =1\}$ is called the \emph{numerical range} of $M$ (\cite{gr}, \cite{pt}); it is always a convex subset of $\CC$ (\cite{gr}, \cite{pt}) and
the main aim of this paper is to explore the convexity for other $(L,K,\sigma)$, sometimes with strong restrictions on the matrix $M$. Under mild assumptions on $L$ and $K$ we have $M=M^\dagger$ if and only if $\Nm (M)\subseteq K$ (Proposition \ref{b1.00}), but this is not always true (Example \ref{b1.02}). Fix an integer $k>0$ and $M_1,\dots ,M_k\in M_{n,n}(L)$. Let $\nu _{M_1,\dots ,M_k}: L^n\to L^k$ be the map
defined by the formula $\nu _{M_1,\dots ,M_k}(u) =(\langle u,M_1u\rangle ,\dots ,\langle u,M_u\rangle )$; we call it the \emph{joint numerical map} of $M_1,\dots ,M_k$. When $L=\CC$ the image
of the map $ \nu _{M_1,\dots ,M_k}|C_n(1)$ is called the \emph{joint numerical range} of the matrices $M_1,\dots ,M_k$ (\cite{lp}).
In \cite{bal} we introduced the following subsets of $K$. Let $\Delta \subseteq L$ be the set of all $\langle a,a\rangle$, $a\in L$. Since $\sigma (\langle a,a\rangle )=
\langle a,a\rangle$ for any $a\in L$, we have $\Delta \subseteq K$. Note that $0\in \Delta$, that $\Delta \setminus \{0\}$ is a multiplicative group and that $\Delta$ is the image of the norm map $L\to K$. Since $\langle a,a\rangle =a^2$ for all $a\in K$, $\Delta$ contains all squares of elements of $K$. For each $n\ge 1$
let $\Delta _n\subseteq K$ be the sum on $n$ elements of $\Delta$. The set $\Delta _n$ is the set of all $\langle u,u\rangle $ for some $u\in L^n$. We have $\Delta _n +\Delta _m =\Delta _{n+m}$ for all $n>0$, $m>0$ and it is easy to check that $\Delta _n\setminus \{0\}$ is a multiplicative group (\cite[Lemma 2]{bal}). 

Take $a,b\in L$ such that $a\ne b$. The \emph{$\Delta$-convex hull} of $\{a,b\}$ is the set of all $ta+(1-t)b$ with $t\in \Delta \cap (1-\Delta)$. At least if $\mathrm{char}(K) =0$, $\Delta \cap (1-\Delta)$ is infinite (\cite[Lemma 8]{bal}). A set $S\subseteq L$ is said to be
{$\Delta$-convex} if for all $a,b\in S$ with $a\ne b$, $S$ contains the $\Delta$-convex hull of $\{a,b\}$. For any $S\subseteq L$ the $\Delta$-convex hull of $S$ is the intersection
of all $\Delta$-convex subsets of $L$ containing $S$, i.e. the smallest $\Delta$-convex subsets of $L$ containing $S$. If we take $\Delta _n\cap (1-\Delta _n)$ instead of $\Delta \cap (1-\Delta)$ we get the notion of $\Delta _n$-convexity.

We met the following obstacles to get the convexity of $\Nm (M)$ and to study the
numerical range, the joint numerical range, the numerical map and the joint numerical map.

\begin{enumerate}
\item $L$ may be not algebraically closed and hence there are matrices $M\in M_{n,n}(L)$, $n\ge 2$, with eigenvalues not contained in $L$;
\item there are $u\in L^n$, $n\ge 2$, with $u\ne 0$ and $\langle u,u\rangle =0$;
\item there are $u\in L^n$ such that $\langle u,u\rangle \ne 0$, but there is no $t\in L$ with $\langle tu,tu\rangle =1$;
\item $\Delta _n$ may strictly contain $\Delta$.
\end{enumerate}

The obstacles (3) and (4) are equivalent (Remark \ref{pp1}). It is very restrictive to assume that $L$ is algebraically closed, but the only case in which we get $\Delta$-convexity of all numerical ranges (Theorem \ref{i1}) requires that $K$ is a real closed field and so $L$ is algebraically closed (\cite[Theorem 1.2.2]{bcr}). When $L$ is not algebraically closed we may
at least get informations for some matrices, e.g. the ones with all eigenvalues contained in $L$, plus some other conditions are satisfied. We discussed (2), (3) and (4) in \cite{bal}. Here we describe another case in which $\Delta _n =\Delta$ for all $n$ and $\langle u,u\rangle =0$ for some $u\in L^n$ if and only if $u=0$ (Proposition \ref{aa4}).
We say that $\langle \ ,\ \rangle$ is \emph{definite up to $n$} if $\langle u,u\rangle \ne 0$ for all $u\in L^n\setminus \{0\}$.

When $L$ and $K$ are finite fields, say $K=\FF_q$, $L= \FF _{q^2}$ and $\sigma$ the Frobenius map $t\mapsto t^q$, we have $\Delta =K$ and so $\Delta _n=\Delta$ for all $n>0$, so the third condition is always satisfied, but the only non-empty $\Delta$-convex subsets of $L$ are the singletons, $L$ and the affine $K$-lines of $L$ seen as a $2$-dimensional $K$ vector space; there are plenty of examples (\cite{cjklr}) of matrices $M$ with $\sharp (\Nm (M))\notin \{1,q,q^2\}$ and hence with numerical range not $\Delta$-convex. The obstacle (2) may be overcome
in all cases in which we have some definite positive conditions (e.g. if $K$ is a subfield of $\RR$ and $L=K(i)$ with $\sigma$ the complex conjugation).
We axiomatize the non-existence of obstacle (2) in the following way.
We say that $\langle \ ,\ \rangle$ is \emph{definite of to $n$} if $\langle u,u\rangle \ne 0$ for all $u\in L^n$ with $u\ne 0$.  In the set-up of (2) and (3) even if $M$ has an eigenvalue $c\in L$ we may have $c\notin \Nm (M)$ (e.g. see Example \ref{d3}; there are many examples over finite fields (\cite{cjklr}).

The classical text-book proof of the Toeplitz-Hausdorff theorem on the convexity of $\Nm (M)$ uses a reduction to the case in which the numerical range is either a segment
or a region bounded by an ellipse with one or two foci (\cite[\S 1.1]{gr}). We axiomatize this case in the following way.

\begin{definition}\label{d1}
A \emph{$\Delta _n$-ellipse with a unique focus at $0$} is a subset $S\subseteq L$ such that there are $\delta _1\in \Delta _n\setminus \{0\}$ and $\delta _2\in \Delta _n\setminus \{0\}$ with $S = \{\sigma (x)y \mid (x,y)\in L^2, \delta _1x\sigma (x)+\delta_2y\sigma (y)=1\}$. In  this case we say that $S$ is a $\Delta _n$-ellipse with a unique
focus with parameters $(\delta _1,\delta _2)$. A \emph{$\Delta _n$-ellipse with two foci} is a subset $S\subseteq L$ such that there are $\delta _1,\delta _2\in \Delta _n\setminus \{0\}$ and $d_1,d_2\in L^\ast$ with
$S =\{d_1y\sigma (x)+d_2 y\sigma (y)\mid (x,y)\in L^2, \delta _1x\sigma (x)+\delta _2y\sigma (y) =1\}$; $(\delta _1,\delta _2,d_1,d_2)$ are the parameters of the ellipse with $2$ foci. A $\Delta _n$-ellipse is set $S\subseteq L$ such
that there are $a, b\in L$ with $b\ne 0$ and $(S-a)/b$ is a $\Delta _n$-ellipse with one or two foci. We define in the same way the $\Delta$-ellipses.
\end{definition}

When $K$ is a finite field, a $\Delta$-ellipse is not $\Delta$-convex. When $n=2$ we find some cases in which $\Nm (M)$ is an ellipse with
one or two foci (Propositions \ref{aa6}, \ref{aa7}). We show some cases in which for any $a, b\in \Nm (M)$ the set $\Nm (M)$ contains either
the $\Delta$-convex hull of $\{a,b\}$ or a $\Delta$-ellipse with $2$ foci containing $\{a,b\}$ (Propositions \ref{aa6}, \ref{aa7} and \ref{aa9}). Easy examples show that $\Nm (M)$ may not contain the spectrum of $M$
(not even the eigenvalues of $M$ contained in $L$), but it obviously contains an eigenvalue $a\in L$ if there is $u\in L^n$ with $\langle u,u\rangle \in \Delta \setminus \{0\}$ and $Mu =au$. Example \ref{d3} shows that it is not sufficient to assume $\langle u,u\rangle \in \Delta _2\setminus \{0\}$.

We may generalize this observation to two eigenvalues in the following way.

\begin{theorem}\label{d2}
Assume that $\langle \ ,\ \rangle $ is definite up to $n$ and take $M\in M_{n,n}(L)$. Fix eigenvalues $a,b\in L$ of $M$ with $a\ne b$ and assume the existence
$u,v\in L^n$ such that $\langle u,u\rangle \in \Delta \setminus \{0\}$, $\langle v,v\rangle  \in \Delta \setminus \{0\}$, $Mu=au$ and $Mv=bv$. 

\quad (a)  If $\langle u,v\rangle =0$, then $\Nm (M)\supseteq \{ta+(1-t)b\} _{t\in \Delta \cap (1-\Delta )}$.

\quad (b) If $\langle u,v\rangle \ne 0$, then there
is $S\subseteq \Nm (M)$ such that $a\in S$, $b\in S$ and $S$ is a $\Delta _n$-ellipse.
\end{theorem}

Example \ref{d3} shows that we cannot use $\Delta _2$ instead of $\Delta$; in Example \ref{d3} we get a $\Delta _2$-ellipse $S \subseteq \Nm (M)$, but $a\notin S$
and $b\ne S$, because $a\notin \Nm (M)$ and $b\notin \Nm (M)$.

In section \ref{S3} we consider the case in which $K$ as an ordering (\cite[Ch. 1]{bcr}). We prove the following results.

\begin{proposition}\label{aa4}
Take $K$ with an ordering $<$ such that each positive element of $K$ is a square. Take $L:= K(i)$ with $\sigma$ induced by the map $i\mapsto -i$. Take $M\in M_{n,n}(L)$ such that $M = M^\dagger$. Then $\Nm (M)$ is $\Delta$-convex.
\end{proposition}

\begin{proposition}\label{aa9}
Take $K$ with an ordering $<$ such that each positive element of $K$ is a square. Take $L:= K(i)$ with $\sigma$ induced by the map $i\mapsto -i$. Take $M\in M_{n\times n}(K)$ and $a, b\in \Nm (M)$ with $a\ne b$.
Then there is $S\subseteq \Nm (M)$ such that $a\in S$, $b\in S$ and $S$ is either the $\Delta$-convex hull of $a$ and $b$ or an ellipse with two foci containing $\{a,b\}$.
\end{proposition}

\begin{theorem}\label{i1}
Assume that $K$ is a real closed field and that $L =K(i)$ with $\sigma$ the complex conjugation. For any $M\in M_{n,n}(L)$ the set $\Nm (M)$ is a closed and bounded
$\Delta$-convex subset of $L$.
\end{theorem}
 
 As an immediate corollary of Theorem \ \ref{i1} we get the following result.

\begin{corollary}\label{i2}
Assume that $K$ is a real closed field and that $L =K(i)$ with $\sigma$ the complex conjugation.  Fix $M, N\in M_{n,n}(L)$ such that
$M^\dagger =M$ and $N^\dagger =N$. Then the joint numerical range $\Nm (M,N)\subset L^2$ is $\Delta$-convex. If $K$ is real closed,
then $\Nm (M,N)$ is a closed, bounded and $\Delta$-convex subset of $K^2$
\end{corollary}

\section{Preliminaries}
For any $n>0$ let $\II _{n\times n}$ denote the identity $n\times n$ matrix (over any field). For any field $F$ set
$F^\ast:= F\setminus \{0\}$. Note that $\Delta _2 =\Delta$ if and only if $\Delta _n=\Delta$ for all $n>1$.

\begin{remark}\label{pp1}
Take $d\in \Delta \setminus \{0\}$ and $u\in L^n$ such that $\langle u,u\rangle =d$. Since $\Delta \setminus \{0\}$ is a multiplicative group, there is $t\in L$ with
$t\sigma (t) =1/d$. Note that $\langle tu,tu\rangle =1$. Fix $a\in \Delta _n\setminus \{0\}$. By assumption there is $v\in L^n$ such that $\langle v,v\rangle =0$.
Assume the existence of $k\in L$ such that $\langle kv,kv\rangle =1$. We get $a =1/c$ with $c:= k\sigma (k)\in \Delta \setminus \{0\}$. Since  $\Delta \setminus \{0\}$ is a multiplicative group, we get $a\in \Delta$.
\end{remark}

\begin{remark}\label{u1}
For any $M\in M_{n,n}(L)$ and any $u, v\in L^n$ we have $\langle u,Mv \rangle = \langle M^\dagger u, v\rangle = \sigma (\langle v,M^\dagger u\rangle)$.
Thus $\Nm (M)\subseteq K$ if $M^\dagger =M$.
\end{remark}

 We fix an element $\beta \in L\setminus K$ with the following property.

First assume $\mathrm{char}(K)\ne 2$. We take as $\beta$ one of the roots of a polynomial $t^2-\alpha$, with $\alpha \in K$, $\alpha $ not a square in $K$ and
$L\cong K[t]/(t^2-\alpha )$; note that $\sigma (\beta )=-\beta$ in this case. We have $L=K+K\beta$ as a $K$-vector space and we see $L^n =K^n+\beta K^n$ as a $2n$-dimensional
$K$-vector space. For any $z = x+y\beta \in L$,
set $\Re z := x$ and $\Im z = y$. We have $x = (z+\sigma (z))/2$ and $y = (z -\sigma (z))/2\beta$. The $K$-linear maps $\Re : L\to K$ and $\Im : L\to K$ depends on the choice of $\beta$. Take any $M\in M_{n,n}(L)$ and set $M_+:= (M+M^\dagger )/2$ and $M_-:= (M-M^\dagger )/2\beta$. We obvious have $M = M_+\beta M_-$
and $M_+^\dagger =M_+$. Since $\sigma (1/2\beta ) =-1/(2\beta )$ we have $M_-^\dagger = M_-$. Remark \ref{u1} gives $\Nm (M_+)\subseteq K$ and $\Nm (M_-)\subseteq K$.
For any $u \in L^n$ we have $\langle u,Mu\rangle = \langle u,M_+u\rangle +\beta \langle u,M_-u\rangle$. Hence the maps $\Re$ and $\Im$
induces surjections $ \Nm (M)\to \Nm (M_+)$ and $\Nm (M) \to \Nm (M_-)$.

Now assume $\mathrm{char}(K)=2$. There is $\epsilon \in K$ such that
the polynomial $t^2+t+\epsilon$ is irreducible in $K$, while $L \cong K[t]/(t^2+t+\epsilon)$. We take as $\beta$ one of the roots in $L$ of $t^2+t+\epsilon$. Note that
$\beta +1$ is a
root of $t^2+t+\epsilon$. Thus $\sigma (\beta )=\beta +1$. We see $L =K+K\beta$ as a $2$-dimensional $K$-vector space and hence $L^n$ as a $2n$-dimensional $K$-vector space and $M_{n,n}(L)$ as a $2n^2$-dimensional
$K$-vector space. If $z=x+y\beta \in L$ with $x,y\in K$, then set $\Re z := x$ and $\Im z:= y$. We have $\sigma (z) = x+y +y\beta$ and hence $y= z+\sigma (z)$
and $x = z -\beta z -\beta \sigma (z) = \sigma (\beta )z + \beta \sigma (z)$. The maps $\Re :L\to K$ and $\Im :L\to K$ are $K$-linear. For any $M\in M_{n,n}(L)$ set $M_+:= (\beta +1)M +\beta M^\dagger$ and $M_-= M +M^\dagger$. Obviously $M_-$ is Hermitian. Since $2\beta =0$, we have
$M = M_+ +\beta M_-$. Since $\sigma (\beta +1)=\beta$ and $\sigma (\beta ) =\beta +1$, $M_+$ is Hermitian. Thus the map $z\mapsto \Re z$ (resp. $z\mapsto \Im z$)
induces surjections $\Nm (M)\to \Nm (M_+)\subseteq K$ (resp. $\Nm (M) \to \Nm (M_-)\subseteq K$).

\begin{lemma}\label{aa1}
Assume that $\langle \ ,\ \rangle$ is definite up to $n$. Take linearly independent $w_1,\dots ,w_m\in L^n$, $m\le n$. Then there
are $f_1,\dots ,f_n\in L^n$ such that $\langle f_i,f_i\rangle \ne 0$ for all $i$, $\langle f_i,f_j\rangle =0$ for all $i\ne j$ and $f_1,\dots ,f_m$ span the linear subspace spanned by $w_1,\dots ,w_m$. If $\Delta =\Delta _n$, then we may find $f_1,\dots ,f_n$ with the additional condition that $\langle f_i,f_i\rangle =1$.
\end{lemma}

\begin{proof}
We use induction on $m$. Call $W$ the linear span of $w_1,\dots ,w_{m-1}$; if $m-1>0$ we take $f_1,\dots ,f_{m-1}$ mutually orthogonal and spanning $W$.
Set $W^\bot =\{u\in L^n\mid \langle u,v\rangle =0$ for all $v\in W\}$. Since $W$ is non-degenerate, we
have $\dim W^\bot = n-m+1$. Since $\langle \ ,\ \rangle$ is definite positive up to $n$, we have $W\cap W^\bot =\{0\}$, i.e. $L^n = W\oplus W^\bot$ (orthogonal decomposition.
Write $f_m = u+v$ with $u\in W$ and $w\in W^\bot$. Since  $w_1,\dots ,w_m$ are linearly independent, we have $v\ne 0$.
Take $f_m=v$. If $m=n$, then we stop. Now assume $m<n$. Let $V$ be the linear span of $w_1,\dots ,w_m$. Set $V^\bot =\{u\in L^n\mid \langle u,v\rangle =0$ for all $v\in V\}$.
We again have $L^n = V\oplus V^\bot$ and we take as $f_{m+1}$ any non-zero element of $V^\bot$. If $m+1<n$ we continue using the orthogonal of the linear span
of $f_1,\dots ,f_{m+1}$. Now assume $\Delta = \Delta _n$. We assumed exactly the conditions for which the usual Gram-Schmidt orthonormal process works; for instance
if $f_1,\dots ,f_{m-1}$ are orthonormal, we have $v = w_m -\sum _{1}^{m-1} \langle w_m,f_i\rangle$; since $\Delta =\Delta _n$ and $\Delta _n\setminus \{0\}$ is a multiplicative group there is $t\in L$ such
that $t\sigma (t) = 1/\langle w,w\rangle$ and we take $f_m:= tv$.
\end{proof}

\begin{lemma}\label{aa4.0}
Assume $\Delta _2=\Delta$. Take $e, f\in L$ such that $e\ne f$ and take $a, b$ in the $\Delta$-convex hull $S$ of $e, f$. Then $S$ contains the $\Delta$-convex hull of $a, b$.
\end{lemma}

\begin{proof}
Take $t_1, t_2, t\in (\Delta \cap (1-\Delta )$ such that $a =t_1e+(1-t_1)f$, $b= t_2e +(1-t_2)f$. We have $ta+(1-t)b = [tt_1 +(1-t)t_2]e +[(1-t)t_1+t(1-t_2)]f$.
Since $\Delta$ is multiplicative and $\Delta _2=\Delta$, we have $ [tt_1 +(1-t)t_2]\in \Delta$ and $[(1-t)t_1+t(1-t_2)]\in \Delta$. Hence $[(1-t)t_1+t(1-t_2)] \in (\Delta \cap (1-\Delta )$.
\end{proof}

If we need algebraic extensions of $L$ (e.g. because some of the eigenvalues of the matrix $M$ are not in $L$) the following set-up may be useful. Let $K$ be a perfect field (e.g. assume $\mathrm{char}(K)=0$). Fix an algebraic closure $\overline{L}$ of $L$ with a fixed inclusion $j: L\hookrightarrow \overline{L}$. The map $j\circ \sigma :L: \to \overline{L}$
extends to a field isomorphism $\sigma ': \overline{L} \to \overline{L}$ with $\sigma ' (x) =x$ for all $x\in K$ (\cite[Theorem V.2.8]{lang}). We fix one such $\sigma '$. For instance, if $K\subseteq \RR$, $L=K(i)$ and $\sigma$ is the complex conjugate, we may just take as $\sigma '$ the complex conjugate ; if $K=\FF _q$ and
$L =\FF _{q^2}$ with $\sigma$ the Frobenius map $t\mapsto t^q$ we may take as $\sigma ': \overline{\FF}_{q^2}\to \overline{\FF}_{q^2}$ the map $t\mapsto t^q$. As seen in the last
example $\sigma '^2$ may not be the identity (in this example $\FF _{q^2}$ is exactly the fixed point set of $\sigma '^2$ and $\sigma '$ has not finite order).
We fix one $\sigma '$ and use it to define
the $K$-bilinear map $\langle \ ,\ \rangle _{\sigma '}: \overline{L}^n\times \overline{L}^n\to \overline{L}$ by the formula
$\langle (u_1,\dots ,u_n),(v_1,\dots ,v_n)\rangle = \sum _{i=1}^{n} \sigma '(u_i)v_i$. We say that  $\langle \ ,\ \rangle _{\sigma '}$
is definite up to $n$ if $\langle u,u\rangle _{\sigma '} \ne 0$ for all $u\in \overline{L}^n\setminus \{0\}$. This is always the case if $K\subset \RR$, $L=K(i)$
and $\sigma$ and $\sigma '$ are induced by the complex conjugation.

\section{General results}\label{S2}

\begin{lemma}\label{b1.001}
Assume that either $\mathrm{char}(K) =0$ or $\mathrm{char}(K) \ne 2$ and there is $(w,\delta )$ such that $\delta \in (\Delta \cap (1-\Delta) \setminus \{0,1\}$, $w\in L\setminus K$, $w\sigma (w)=\delta$ and both $\delta $ and $1-\delta $ are squares in $K$.
Take $M\in M_{n,n}(L)$ such that $M = M^\dagger$ and $\sharp (\Nm (M)) =1$. Then $M = c\II _{n\times n}$ for some $c\in K$.
\end{lemma}

\begin{proof}
If $M$ is Hermitian, then $\Nm (M)\subseteq K$ by Remark \ref{u1}. Assume $\Nm (M) =\{c\}$ for some $c$. If $\mathrm{char}(K) =0$, then apply \cite[Proposition 1]{bal}.
Hence we may assume the existence of $(w,\delta)$. Since $c\in K$, $M-c\II _{n\times n}$ is Hermitian. Thus taking  $M-c\II _{n\times n}$ instead of $M$ we reduce to the case
$c=0$. Write $M =(m_{ij})$, $i,j=1,\dots ,n$. Since $m_{ii} = \langle e_i,Me_i\rangle$, all the diagonal elements of $M$ are zeroes. Assume the existence of $1\le i < j\le n$
with $m_{ij} \ne 0$. We have $m_{ji}= \sigma (m_{ij})$. Take $u = xe_i+ye_j$ with $x\sigma (x) +y\sigma (y)=1$. Since $m_{ii} = m_{jj} =0$, we have
$0 = \langle u,Mu\rangle = m_{ij}\sigma (x)y + \sigma (m_{ij})x\sigma (y) =0$. Hence $e+\sigma (e)  =0$, where $e:= m_{ij}\sigma (x)y$. Since $\mathrm{char}(K) \ne 2$,
there is $\alpha \in K$ such that $e = \alpha\beta$. Take $x, y\in K$ with $x^2=\delta $ and $y^2 =1-\delta$. Since $\delta \notin \{0,1\}$, we have
$xy\ne 0$. Since $e\in K\beta$, we get $m_{ij} =k\beta$ for some $k\in K^\ast$. Thus $\sigma (x)y \in K$ for all $(x,y)\in L^2$ with $x\sigma (x)+y\sigma (y) =1$.
Taking $x=w$ and $y\in K$ with $y^2=1-\delta$ we get a contradiction.\end{proof}

\begin{proposition}\label{b1.00}
Assume that either $\mathrm{char}(K) =0$ or $\mathrm{char}(K) \ne 2$ and there is $(w,\delta )$ such that $\delta \in (\Delta \cap (1-\Delta) \setminus \{0,1\}$, $w\in L\setminus K$, $w\sigma (w)=\delta$ and both $\delta $ and $1-\delta $ are squares in $K$. A matrix $M\in M_{n,n}(L)$ is Hermitian if and only if $\Nm (M)\subseteq K$.
\end{proposition}

\begin{proof}
If $M$ is Hermitian, then $\Nm (M)\subseteq K$ by Remark \ref{u1}. For an arbitrary $M\in M_{n,n}(L)$ write $M = M_++\beta M_-$. $M$ is Hermitian if and only if
$M_-=0\II _{n\times n}$. Since the map $\Im : \Nm (M)\to \Im (M_-)$ is surjective, $\Nm (M)\subseteq K$ if and only if $\Nm (M_- )=\{0\}$. Apply Lemma \ref{b1.001}.
\end{proof}

\begin{example}\label{b1.02}
Take $K=\FF _2$, $L=\FF _4$ and as $\sigma$ the Frobenius map $t\to t^2$. If $u= (x,y)\in L^2$ we have $\langle u,u\rangle =1$ if and only if $x^3+y^3=1$,
i.e. (since $z^3=1 $ if $z\in \FF _4^\ast$) if and only if $(x,y)\in \{(0,1),(1,0)\}$. Hence for each $M\in M_{2,2}(\FF _4)$ the numerical range $\Nm (M)$ is the set of all diagonal
elements of $M$. Hence there are $16$ matrices $M\in M_{2,2}(\FF _4)$ with $\Nm (M)=\{0\}$ and $4$ of them are Hermitian. This is the only example we know (and the only one for finite fields).
\end{example}

\begin{proof}[Proof of Theorem \ref{d2}:]
Taking a multiple of $u$ and $v$ instead of $u$ and $v$ and applying Remark \ref{pp1} we reduce to the case
$ \langle u,u\rangle = \langle v,v\rangle =1$. Since $a\ne b$, $u$ and $v$ are linearly independent. Taking $\frac{1}{b-a}(M-a\II _{n\times n})$ instead of $M$
we reduce to the case $a=0$ and $b=1$. Thus $u\ne 0$, $Mu =0$ and $Mv=v$.

\quad (a) First assume assume $\langle u,v\rangle =0$ and hence $\langle v,u\rangle =0$. Fix $\delta \in \Delta \cap (1-\Delta )$
and write $\delta =y\sigma (y)$ and $1-\delta =x\sigma(x)$ for some $x,y\in L$. Take $m= xu+yv$. We have $\langle m,m\rangle =1$, because $x\sigma (x)+
y\sigma (y) =1$. We have $Mm = yv$ and hence $\langle m,Mm \rangle =\sigma (y)y =\delta$.

\quad (b) Now assume $d:= \langle u,v\rangle \ne 0$. Set $w:= v -\langle v,u\rangle u$. Since $u, w$ and $u, v$ have the same linear span, we have $w\ne 0$.
Since $\langle \ ,\ \rangle$ is non-degenerate up to $n$, we have $c:= \langle w,w\rangle \in \Delta _n\setminus \{0\}$. We have $\langle w,u\rangle =0$.
Take $m = xu+yw$. We have $\langle m,m\rangle =1$ if and only if $\sigma (x)x +cy\sigma (y) =1$. We have $Mm = yv = yw + ydu$
and hence $\langle m,Mm \rangle =  d^2\sigma (x)y + c\sigma (y)y$. The latter set is a $\Delta _n$-ellipse with $2$ foci.
\end{proof}

\begin{example}\label{d3}
Assume that $\langle \ ,\ \rangle $ is definite up to $2$, but $\Delta _2\ne \Delta$ and fix $\delta \in \Delta _2\setminus \Delta$. We may take $K=\QQ$, $L = \QQ (i)$
and $\sigma$ the complex conjugate (\cite[Example 2]{bal}). Note that $\delta \ne 0$. Fix $u\in L^2$ such that $\langle u,u\rangle =\delta$.
Since $\langle \ ,\ \rangle$ is non-degenerate, there is $m\in L^2$ such that $\langle m,u\rangle =0$ and $m\ne 0$. Let $M\in M_{2,2}(L)$ be the only matrix with
$Mu =0$ and $Mm=m$. Hence $0$ and $1$ are the eigenvalues of $M$. Set $\delta _1:= \langle m,m\rangle$. We claim that $0\notin \Nm (M)$. Take $v =xm+yu$. Since $Mv =xm$ and $m, u$ are orthogonal, we have $\langle v,Mv\rangle =0$
if and only if $x=0$. We have $\langle v,v \rangle = \sigma (x)x\delta \ne 1$ (Remark \ref{pp1}). Now
we check that $1\notin \Nm (M)$. We have $\langle v,v\rangle =1$ if and only if $\delta _1 x\sigma (x) +\delta y\sigma (1)=1$. We have $\langle v,Mv\rangle =1$
if and only if $\delta _1x\sigma (x) =1$, i.e. if and only if $y=0$. Hence $1\in \Nm (M)$ if and only if there is $x\in L$ with $\delta _1 x\sigma (x) =1$, i.e.
if and only if $\langle f_1,f_1\rangle =1$, where $f_1 =xm$. Write $f_1=g_1e_1+g_2e_2$ and set $f_2:= g_2e_1+g_1e_1$. We get $\langle f_2,f_2\rangle =1$
and $\langle f_2,f_1\rangle =0$. Hence there is $t\in L^\ast$ with $u=tf_2$. We get $\delta =t\sigma (t)\in \delta$, a contradiction.
This situation is satisfied if we take $K=\QQ$, $L=\QQ (i)$, $\sigma$ and $\sigma '$ the complex conjugate. Note that in this case ${\sigma '}^2$ is the identity
and $\langle \ ,\ \rangle _{\sigma '}$ is definite for all $n$. Hence we do not see a reasonable way of weaking the assumptions in Theorem \ref{d2}.
\end{example}

\begin{lemma}\label{b2}
Fix an integer $n\ge 2$ and assume that $\langle \ ,\ \rangle _{\sigma '}$ is definite up to $n$ and take $M\in M_{n,n}(L)$ such that $M^\dagger =M$. Let $c\in \overline{L}$
be any eigenvalue of $M$. Then $\sigma '({c}) =c$.
\end{lemma}

\begin{proof}
Take $u\in \overline{L}^n$, $u\ne 0$, such that $Mu = cu$.
We have $c \langle u,u\rangle = \langle u,Mu\rangle = \langle Mu,u\rangle =\sigma '({c})\langle u,u\rangle$. 
\end{proof}

\begin{lemma}\label{b3}
Fix an integer $n\ge 2$ and assume that $\langle \ ,\ \rangle _{\sigma '}$ is definite up to $n$ and take $M\in M_{n,n}(L)$ such that $M^\dagger =M$.
Take eigenvalues $c, d\in \overline{L}$ of $M$ such that $c\ne d$ and any $u, v\in \overline{L}$ such that $Mu =cu$ and $Mv =dv$.
Then $\langle u,v\rangle _{\sigma '} =0$.
\end{lemma}

\begin{proof}
We have $d \langle u,v\rangle_{\sigma '} = \langle u,Mv\rangle_{\sigma '} = \langle Mu,v\rangle_{\sigma '} =\sigma '({c})\langle u,v\rangle_{\sigma '}$. Lemma \ref{b2} gives $\sigma '({c}) = c\ne d$.
\end{proof}

\begin{remark}\label{bb1}
Take $A\in M_{n,n}(L)$, $B\in M_{m,m}(L)$ and set $M:= A\oplus B\in M_{n+m,n+m}(L)$. $\Nm (M)$ is the $\Delta$-convex hull
of $\Nm (A)$ and $\Nm (B)$ (\cite[Lemma 7]{bal}.
\end{remark}

\begin{lemma}\label{aa5}
Take $n=2$ and fix $b\in L^\ast$. Take
\[M = \begin{pmatrix}
0 & b \\
0 & 1
\end{pmatrix}\]
Then $\Nm (M)$ is the ellipse with foci at $\{0,1\}$ and parameters $(\delta _1,\delta _2, d_1,d_2) =(1,1,b,1)$.
\end{lemma}

\begin{proof}
Take $u=xe_1+ye_2$. We have $\langle u,u\rangle =1$ if and only if $x\sigma (x)+y\sigma (y)=1$ and $\langle u,Mu\rangle = bx\sigma (y) +y\sigma (y)$.
\end{proof}

\begin{lemma}\label{aa5.0}
Take $n=2$ and the matrix
\[M = \begin{pmatrix}
0 & 1 \\
0 & 0
\end{pmatrix}\]
Then $\Nm (M)$ is the ellipse with one focus at $0$ and $\delta _1=\delta _2=1$.
\end{lemma}

\begin{proof}
Take $u=xe_1+ye_2$. We have $\langle u,u\rangle =1$ if and only if $x\sigma (x)+y\sigma (y)=1$ and $\langle u,Mu\rangle = \sigma (x)y$.
\end{proof}

\begin{proposition}\label{aa6}
Assume that $\langle \ ,\ \rangle$ is definite up to $2$ and take $M\in M_{2,2}(L)$ with two different eigenvalues, both in $L$. Assume either $\Delta _2=\Delta$ or that
$M$ has at least one eigenvector $u$ with $\langle u,u\rangle \in \Delta$. Then $\Nm (M)$ is an ellipse with two foci or the $\Delta$-convex hull of its eigenvalues.
\end{proposition}

\begin{proof}
If $\Delta _2=\Delta$, then $\langle w,w\rangle \in \Delta \setminus \{0\}$ for any $w\in L^2\setminus \{0\}$, because $\langle \ ,\ \rangle$ is definite up to $2$. Call $c_1$ and $c_2$ the two eigenvalues of $M$ with $c_1$ with eigenvalue $u$ with $\langle u,u\rangle \in \Delta \setminus \{0\}$. Take $t\in L^\ast$ such that
$t\sigma (t) = 1/\langle u,u\rangle$ and set $f_1:= tu$. Write $u = a_1e_1+a_2e_2$ with $a_i\in L$ and set $f_2:= a_2e_1-a_1e_2$. We have $\langle f_i,f_j\rangle =0$
for all $i\ne j$. We have $\langle f_2,f_2\rangle =\langle f_1,f_1\rangle =1$. Taking $\frac{1}{c_2-c_1}(M-c_1\II _{2\times 2})$ instead of $M$ we reduce to the case $c_1=0$
and $c_2=1$. Since $f_1$ and $f_2$ is an orthonormal basis, we may use it instead of $e_1$ and $e_2$ to compute $\Nm (M)$. If $f_2$ is an eigenvector
of $M$, then it has eigenvalue $1$ and hence $M$ is unitarily equivalent to  $0\II _{1\times 1}\oplus \II _{1\times 1}$ and hence $\Nm (M) =\Delta \cap (1-\Delta )$ (the $\Delta$-convex hull of $0$ and $1$).
\end{proof}

\begin{proposition}\label{aa7}
Assume either $\mathrm{char}(L)\ne 2$ or that $L$ is perfect. Assume that $\langle \ ,\ \rangle$ is definite up to $2$ and take $M\in M_{2,2}(L)$ with a unique eigenvalue, $c$. Then $c\in L$. Assume also
that either $\Delta _2=\Delta$ or that $M$ has an eigenvector $u$ with $\langle u,u\rangle \in \Delta$. Then either $M =c\II _{2\times 2}$ and
$\Nm (M) =\{c\}$ or $\Nm (M)$ is an ellipse with one focus.
\end{proposition}

\begin{proof}
The characteristic polynomial of $M$ has degree $2$ and $c$ is its only root over $\overline{L}$. We have $c\in L$, because in the case $\mathrm{char}(L)=2$ we assumed that $L$ is perfect. Taking $M-c\II _{2\times 2}$ instead of $M$ we reduce to the case $c=0$. By assumption $M$ has an eigenvector $u$ with $\langle u,u\rangle \in \Delta\setminus \{0\}$. Take $t\in L^\ast$ such that
$t\sigma (t) = 1/\langle u,u\rangle$ and set $f_1:= tu$. Write $u = a_1e_1+a_2e_2$ with $a_i\in L$ and set $f_2:= a_2e_1-a_1e_2$. We have $\langle f_i,f_j\rangle =0$
for all $i\ne j$. We have $\langle f_2,f_2\rangle =\langle f_1,f_1\rangle =1$. Write $M = (a_{ij})$, $i=1,2$, in the orthonormal basis $f_1$
and $f_2$. We have $a_{11}=a_{21}=0$. Since $0$ is the only eigenvalue of $M$, we have $a_{22}=0$.
If $f_2$ is an eigenvector of $M$, then $M = 0\II _{2\times 2}$. Hence we may assume that $a_{12}\ne 0$. Apply Lemma \ref{aa5.0} to $(1/a_{12})M$.
\end{proof}

\section{Ordered fields}\label{S3}

Let $(K,<)$ be any ordered field. Fix a real closure $\Rr \supseteq K$ (\cite[Theorem 1.3.2]{bcr}) and call again $<$ the ordering of $\Rr$ extending the
ordering $<$ of $K$. Set $K_0:= K$. For each $i\ge 1$ let $K_i\subseteq \Rr$ be the field obtained adding to $K _{i-1}$ all the square-roots of the positive (with respect to the ordering $<$ of $\Rr$) elements of $K_{i-1}$. Set $\widetilde{K}:= \cup K_{i\ge 1}$. $\widetilde{K}$ is the smallest subfield of $\Rr$ such that $K\subseteq \widetilde{K}$
and all positive elements of $\widetilde{K}$ are square (\cite[Proposition 16.4]{har}; ordered fields with this property are called \emph{euclidean} in \cite[Proposition 16.2]{har}. Set $\widetilde{L} = \widetilde{K}(i)$ and use the complex conjugation $\sigma '$ to get a positive definite sesquilinear form
up to any $n$. Since each square is contained in $\Delta$, $(\widetilde{L},\sigma )$ has $\Delta =\widetilde{K}_{\ge 0} \supseteq \Delta _n$ and so $\Delta =\Delta _n$ for all $n$.

\begin{proof}[Proof of Proposition \ref{aa4}:]
If $z=x+yi\in L$ with $x,y\in K$, then $\langle z,z\rangle =x^2+y^2\ge 0$ with equality if and only if $x=y=0$. Hence $\langle \ ,\ \rangle$ is definite up to any $m>0$.
Since $\Delta$ contains all squares, we have $\Delta _m =\Delta$ for all $m\ge 2$. If $\Nm (M) =\{c\}$ for some $c\in K$, then $M =c\II _{n\times n}$ by \cite[Proposition 1]{bal}.
Now assume the existence of $a, b\in \Nm (M)$ such that $a\ne b$ and take $u, v\in L^n$ such that
$\langle u,u\rangle = \langle v,v\rangle =1$, $\langle u,Mu\rangle =a$ and $\langle v,Mv\rangle =b$. By Lemma \ref{aa1} we 
may assume that $u=e_1$ and $v=e_2$. Set $N:= M_{|(Le_1+Le_2)}$. We have $a, b\in \Nm (N)$ and $\Nm (N)\subseteq \Nm (M)$. Write $N = (a_{ij})$, $i, j=1,2$. We have $a_{11} =a\in K$, $a_{22}=b\in K$, $a_{21} =\langle v,Mu\rangle = \langle Mv,u\rangle = \sigma (\langle u,Mv\rangle ) = \sigma (a_{12})$. Hence $N^\dagger =N$. Write $a_{12} = x+yi$ with $x, y\in K$. Let $f(t)$ be the characteristic polinomial of $N$. Taking $N -\frac{a_{11}+a_{22}}{2}\II _{2\times 2}$ instead of $N$ we reduce to the case
$a_{11}=-a_{22}$, i.e. $N$ has trace $0$, i.e. $f(t) = t^2-a_{11}^2 -a_{12}\sigma (a_{12}) =t^2-a_{11}^2-x^2-y^2$. Since sums of squares of elements of $K$ are squares by
our assumption on $K$, we have $f(t) =t^2-d^2$ for some $d\in K$. Hence all the eigenvalues of $N$ are contained in $K$. If $d\ne 0$, then $N$
has two different eigenvalues, $d$ and $-d$, and Lemma \ref{b3} gives that $N$ is unitarily equivalent to $(-d)\II _{1\times 1}\oplus d\II _{1\times 1}$. In this case
we apply Remark \ref{bb1}. If $d=0$, then $a_{11}=x=y=0$ and so $N= 0\II _{2\times 2}$. Thus $\Nm (N) =\{0\}$, contradicting the assumption $a\ne b$.
\end{proof}

\begin{proposition}\label{aa4=}
Take $K$ with an ordering $<$ such that each positive element of $K$ is a square. Take $L:= K(i)$ with $\sigma$ induced by the map $i\mapsto -i$. Take $M\in M_{2,2}(K)$.
Then $\Nm (M)$ is either a point, or a $\Delta$-segment or a $\Delta$-ellipse with one or $2$ foci.
\end{proposition}
\begin{proof}
Write $M = (m_{ij})$, $i, j=1,2$. Using $M-\frac{m_{11}+m_{22}}{2}\II _{2\times 2}$ instead of $M$ we reduce to the case
$m_{11}+m_{12} =0$. Hence the characteristic polynomial $f(t)$ of $M$ is of the form $f(t)=t^2+d$ for some $d\in K$. If $d\le 0$ (resp. $d>0$), then it has $2$ roots in $K$ (resp. $L$), because every positive element of $K$ has a square root in $K$. If $d\ne 0$, then we apply Proposition \ref{aa6}. If $d=0$, then either $M =0\II _{2\times 2}$ or we apply Proposition \ref{aa7}.
\end{proof}

\begin{proof}[Proof of Proposition \ref{aa9}:]
Let $\Rr$ be a real closure of $(K,<)$. Since $(K,<)$ is euclidean, we have $\Delta =\Delta _m$ for every $m>0$ and $\langle \ ,\ \rangle _{\sigma '}$ is definite up to any $m>0$. Take $u, v\in L^n$ such that $\langle u,u\rangle = \langle v,v\rangle = 1$, $a=\langle u,Mu\rangle$ and $b= \langle v,Mv\rangle$. Since $a\ne b$, $u$ and $v$
are not proportional. Let $W$ be the linear span $u$ and $v$.  By Lemma \ref{aa1} we may assume $n=2$ with the matrix $N:= M_{|W}$. If $N$ has a unique eigenvalue, $c$,
then we apply Proposition \ref{aa7}. If $N$ has two different eigenvalues, both of them in $L$, we use Proposition \ref{aa7}.
Now assume that $N$ has two eigenvalues $c_i\in \Rr (i)$, $i=1,2$, none of them in $L$. Since the characteristic polynomial of $N$ has
coefficients in $K$, these eigenvalues generate a degree $2$ Galois extension $K_1$ of $K$, i.e. there are a non-square $m\in K$ such that $K_1(g)$ with $g^2=e$.
Since $(K,<)$ is euclidean, $e<0$ and $-e = f^2$ for some $f\in K$. Thus $K_1 =K(i) =L$, a contradiction.
\end{proof}

If $K$ is real closed, then Proposition \ref{aa9} is trivial, because in that case $\Nm (M)$ is a semi-algebraically connect bounded and closed semi-algebraic subset
of $K$ (see Example \ref{da1}), i.e. a closed interval (\cite[Proposition 2.1.7]{bcr}).

\begin{remark}\label{da1}
Take as $K$ a real closed ordered field (\cite[\S 1.2]{bcr}) and take $L= K(i)$ with $\sigma$ the complex conjugation. Thus $L$ is algebraically closed, $\sigma '=\sigma$ and $\sigma ^{'2}$ is the identity.
Therefore for every $n>0$ $\langle \ ,\ \rangle $ and $\langle \ ,\ \rangle _{\sigma '}$ are definite up to $n$. Since $K_{\ge 0}$ is the set of all squares of elements
of $K$, we have $\Delta _n \subseteq K_{\ge 0} =\Delta$ and so $\Delta _n=\Delta$ for all $n$. The set $S^{2n-1}:= \{u\in L^n\mid \langle u,u\rangle  =1\}$ is a closed and bounded subset of $L^n=K^{2n}$ for the euclidean topology and the map $u\mapsto \langle u,Mu\rangle$ is real algebraic. Thus $\Nm (M)$ is a closed and bounded semi-algebraic subset of $L$ (\cite[Proposition 2.5.7]{bcr}). Since the sphere $S^{2n-1}$ is semi-algebraically connected in the sense
of \cite[Definition 2.4.2]{bcr}, $\Nm (M)$ is semi-algebraically connected. By \cite[Proposition 2.5.13]{bcr} $\Nm (M)$ is semi-algebraically path connected. Take $a, b\in L$ with $a\ne b$. The $\Delta$-convex hull
$A\subseteq L$ is the segment $\{ta+(1-t)a\}_{t\in K, 0\le t\le 1}$. Hence $\Delta$-convexity is preserved by any $K$-affine map. Since $L$ is algebraically closed, to prove that
$\Nm (M)$ is convex, it is sufficient to prove that ellipse with one or two foci are $\Delta$-convex, but we prefer to follow the convexity proof given in \cite{pt}. To show that ellipses
with one foci are disks of $L=K^2$ (resp. an ellipse with two foci plus its interior in the sense of the euclidean topology) we may use the following remarks. First of all we reduce
to the case $\delta _1=\delta _2=1$. Under this assumption they are the numerical range of the matrix $M$ appearing in Lemma \ref{aa5.0} (resp. Lemma \ref{aa5}). On $L\setminus \{0\}$ the function $|z| $ is semi-algebraic and so $\Re (z/|z|)$ and $\Im (z/|z|)$. We use this functions instead of the functions $\cos $ and $\sin$ to show
that $\{|z| \le 1/2\} =\Nm (M)$ in the case of Lemma \ref{aa5.0} following the proof in \cite[Example 1, page 1]{gr}. To adapt the proof of \cite[Example 2, pages 2,3 and Lemma 1.1]{gr} to get the $\Delta$-convexity in the set-up of Lemma \ref{aa5} we need
to make sense of some trigonometric expression like $\sin $ and $\cos$; if $K\ne \RR$ we cannot use the real exponential function to get the trigonometric function (\cite{a}). With our substitute of the functions $\cos$ and $\sin$, the first one is even and the second
one is odd. For any fixed $z\in S^1:= \{|z|=1\}$ there are $a,b\in S^1$ with $a^2 =z$ and $b^3=z$. Instead of $\cos (u+v )$ (resp. $\sin (u+v )$) we may use
$\cos (u)\cos (v) -\sin (u)\sin (v)$ (resp. $\sin (u)\cos (v)+\cos (u) \sin (v)$).
\end{remark}

Theorem \ref{i1} easily follows from the following result (\cite[Lemma 2]{pt}; the proof in \cite[page 4]{pt} works with minimal modifications).

\begin{proposition}\label{ii1}
Assume that $K$ is a real closed field and that $L =K(i)$ with $\sigma$ the complex conjugation. Fix $M\in M_{n,n}(L)$ such that $M^\dagger =M$ and any $t\in L$.
Then either $\nu _M^{-1}(t) =\emptyset$ or $\nu _M^{-1}(t)$ is semi-algebraically arc-connected.
\end{proposition}

\begin{proof}
If $t\notin K$, then $\nu _M^{-1}(t) =\emptyset$ by Lemma \ref{b2} and hence we may assume $t\in K$. If $t\in K$, then $M-t\II _{n\times n}$ is Hermitian. Since $\Nm (M-t\II _{n\times n})=
\Nm (M)-t$, taking $M-t\II _{n\times n}$ instead of $M$ we reduce to the case $t=0$. Since $UMU^\dagger$ is Hermitian and $\Nm (UMU^\dagger  )=\Nm (M)$ for every unitary $U$, we reduce to the case in which $M$ is diagonal, say $M =(m_{ij})$ with $m_{ii}\in K$ and $m_{ij}=0$ for all $i\ne j$. Take $a=(a_1,\dots ,a_n)\in \nu _M^{-1}(0)$.
For all $(z_1,\dots ,z_n)\in L^n$ with $\sigma (z_i)z_i =1$ for all $i$ we have $(z_1a_1,\dots ,z_na_n)\in \nu _M^{-1}(0)$, because $M$ is a diagonal matrix.
Since the circle $S^1= \{\sigma (z)z =1\} \subset L$ is semi-algebraically arc-connected, it is sufficient to prove that if $a=(a_1,\dots ,a_n), b=(b_1,\dots ,b_n)\in \nu _M^{-1}(0)$
with $a_i,b_i\in K$ for all $i$, then $a$ and $b$ are connected by a semi-algebraic arc contained in $\nu _M^{-1}(0)$. As in \cite[page 4]{pt} it
is sufficient to use the semi-algebraic arc $u(t) = (u_1(t),\dots u_n(t))$ from the unit interval $[0,1]$ into $\nu _M^{-1}(0)\cap K^n$ with $u_j(t) = \sqrt{(1-t)a_j^2+tb_j^2}$.
\end{proof}

\begin{proof}[Proof of Theorem \ref{i1}:]
(\cite[page 4]{pt}) Fix $a, b\in \Nm (M)$ such that $a\ne b$. Taking $\frac{1}{b-a}(M-a\II _{n\times n})$ instead of $M$ we reduce to the case $a=0$ and $b=1$. Since $\Delta$ is the set of all non-negative elements of $K$, it is sufficient to prove that $\Nm (M)$ contains the closed interval $[0,1]\subset K$. Take $u, v\in L^n$ such that
$\langle u,u\rangle =\langle u,u\rangle =1$, $\langle u,Mu\rangle =0$ and $\langle v,Mv\rangle =1$. Since $\{0,1\}\in K$, we have $\langle u,M_+u\rangle =0$,
$\langle u,M_-u\rangle =0$, $\langle v,M_+v\rangle =1$ and $\langle v,M_--v\rangle =0$. Since $M_-$ is Hermitian, and $\{u,v\}\subset \nu _{M_-}^{-1}(0)$, there
is a semi-algebraic arc $m: [0,1]\to \nu _{M_-}^{-1}(0)$ with $m(0) =u$ and $m(1) =v$ (Proposition \ref{ii1}). We have $\langle m(t),Mm(t)\rangle = \langle m(t),M_+m(t)\rangle$, because
$\langle m(t),M_-m(t)\rangle =0$ for all $t\in [0,1]$. Since $M_+$ is Hermitian, $\langle m(t),M_+m(t)\rangle \in K$ for all $K$. The set $Z$ of all $ \langle m(t),M_+m(t)\rangle$, $t\in [0,1]$, is a
semi-algebraic arc contained in $K$ and hence it is an interval (\cite[Proposition 2.1.7]{bcr}). Since $\{0,1\}\subset Z\subseteq \Nm (M)$, then $[0,1]\subseteq \Nm (M)$.
\end{proof}

\begin{proof}[Proof of Corollary \ref{i2}:]
Identify $L^2$ with $K^2+iK^2$ sending any $z=(z_1,z_2)\in L^2$ to $(\Re z_1, \Re z_2,\Im z_1,\Im z_2)$ and then identify any $(a_1,a_2)\in K^2$ with $a_1+ia_2\in L$.
Take $A:= M+iN\in M_{n,n}(L)$. Since $M$ and $N$ are Hermitian, we have $A_+ =M$, $A_-=N$, $\Nm (M) \subset K$, $\Nm (N)\subset K$ and $c\in \Nm (A)$ if and only if $\Re c\in \Nm (M)$ and
$\Im c \in \Nm (N)$. Apply Theorem \ref{i1} to $A$.
\end{proof}

\providecommand{\bysame}{\leavevmode\hbox to3em{\hrulefill}\thinspace}

\end{document}